\documentclass[preprint,11pt]{elsarticle}
\usepackage[utf8]{inputenc}
\usepackage{amsmath}
\usepackage{enumerate}
\usepackage{bm}
\usepackage[T1]{fontenc}
\usepackage{amsthm}
\usepackage{amsfonts}
\usepackage{amssymb}
\usepackage{graphicx}
\usepackage{enumerate}
\usepackage[english]{babel}
\usepackage{comment}
\newtheorem{theorem}{Theorem}[section]

\newtheorem{proposition}{Proposition}[section]

\newtheorem{remark}{Remark}[section]

\newcommand{\Om} {\Omega}
\newcommand{\pa} {\partial}
\newcommand{\be} {\begin{equation}}
	\newcommand{\ee} {\end{equation}}
\newcommand{\bea} {\begin{eqnarray}}
	\newcommand{\eea} {\end{eqnarray}}
\newcommand{\Bea} {\begin{eqnarray*}}
	\newcommand{\Eea} {\end{eqnarray*}}

\newcommand{\de} {\delta}

\newcommand{\De} {\Delta}
\newcommand{\la} {\lambda}

\newcommand{\noi} {\noindent}

\usepackage{amssymb}
\usepackage{amsthm}
\usepackage{xcolor}

\journal{XXXX}

\begin{document}
	
	\begin{frontmatter}
		
		
		
		\title{Parameter estimates and a uniqueness result for double phase problem with a singular nonlinearity}
		\author[add1]{R.Dhanya }
		\ead{dhanya.tr@iisertvm.ac.in}
		\author[add1]{M.S. Indulekha}
		\ead{indulekhams17@iisertvm.ac.in}
		\address[add1]{School of Mathematics, IISER Thiruvananthapuram, Thiruvananthapuram, Kerala, 695551, India  }

		
		%
%

\begin{abstract}
	We consider the boundary value problem $-\De_p u_\la -\De_q u_\la =\lambda g(x) u_\la
	^{-\beta}$ in $\Om$  , $u_\la=0$ on $\pa \Om$ with $u_\la>0$ in $\Omega.$ We assume $\Om$ is a bounded open set in $\mathbb{R}^N$ with smooth boundary,  $1<p<q<\infty$, $\beta\in [0,1),$ $g$ is a positive weight function and $\lambda$ is a positive parameter. We derive an estimate for $u_\lambda$ which describes its exact behavior when the parameter $\lambda$ is large. In general,  by invoking appropriate comparison principles, this estimate can be used as a powerful tool in deducing the existence, non-existence and multiplicity of  positive solutions of nonlinear elliptic boundary value problems. Here, as an application of this estimate,  we obtain a uniqueness result for a nonlinear elliptic boundary value problem with a singular nonlinearity. 
	
\end{abstract}	
		
		
	\end{frontmatter}
	
\section{Introduction}
 We are interested in the positive solution of the non-homogeneous quasi-linear boundary value problem 
$$
(P_\lambda)\left\{ \begin{array}{rcl}
	-\Delta_p u-\Delta_q u&=& \la h(x,u) \mbox{ in } \Omega \\
	u&= &0 \mbox{ on }{\partial \Omega}
\end{array}\right.
$$
where $\Om$ is a bounded smooth domain in $\mathbb{R}^N$, $1<p<q<\infty$  and $\la>0.$ We say that $u$ is a solution of $(P_\lambda)$ if $u\in W^{1,q}_0(\Omega)$ and satisfies the PDE in the weak sense:
$$\int_{\Om} |\nabla u|^{p-2}\nabla u \nabla \varphi + \int_{\Om} |\nabla u|^{q-2}\nabla u \nabla \varphi =\lambda \int_{\Omega} h(x,u) \varphi \mbox{ for all }\varphi \in C^\infty_c(\Omega)$$
The main focus of this article is to obtain the  asymptotic estimate for the positive solution $u_\lambda$ of $(P_\lambda)$ when the function $h(x,u)$ takes the form $g(x) u^{-\beta}$ for $\beta\in [0,1)$ along with certain conditions on $g(x).$ Moreover, we consider the uniqueness of positive solution of  $(P_\la)$ for large parameter when $h(x,u)=f(u)u^{-\beta}$  and $1< p<q \leq 2.$\\  
Problems of the type $(P_\la)$ are closely associated to the minimization of certain energy functionals of the type 
\begin{equation}
\label{eqn:do-ph-functional}
    u \mapsto \int_\Om |\nabla u|^{p}+a(x)|\nabla u|^{q} dx 
\end{equation}
where $1<p<q<\infty$, $\Om \subset \mathbb{R}^{N}$ is a bounded domain and $a: \mathbb{R} \rightarrow [0, \infty)$ is a measurable function. These functionals are called double phase functionals and were introduced by Zhykov(\cite{zhykov1986averaging}, \cite{jikov2012homogenization}) to model strongly anisotropic materials. The differential operator counterpart of (\ref{eqn:do-ph-functional}), $u \mapsto -\De_p u -a(x)\De_q u$, known as the double phase operator is used in modelling a variety of physical problems in plasma physics\cite{wilhelmsson1987explosive}, biophysics\cite{fife2013mathematical}, reaction diffusion systems \cite{cherfils2005stationary}, etc. and are also well analyzed theoretically. 


   The quasilinear elliptic operator $\mathcal{L}_{p,q}u :=\De_p u +\De_q u $ is  called p-q Laplace operator and it reduces to the standard p Laplace operator when $p=q$. It is observed that several existence and multiplicity results for the p Laplacian can be extended to the p-q Laplacian as well since the associated energy functional preserves the required variational structure (see \cite{marano2017some} and the references therein). But the non-homogeneous nature of the operator poses a greater challenge in understanding certain problems like  the qualitative properties of the solutions of nonlinear p-q Laplace equation, eigenvalue problem of p-q Laplacian etc.  In this work, we mainly address one such difficulty associated to nonhomogenity namely, the scaling property of the solutions of p-q Laplace equation. For instance, if $u_\lambda$ is the unique solution of  $	-\De_p u_\la=\la \text{ in } \Om \text{ and } u_\la=0 \text{ on } \pa\Omega,$ then it can easily be verified that $u_\la=\la^{\frac{1}{p-1}}u_1.$ {A result of this type is impossible for p-q Laplace operator due to its non-homogeneous nature. The novelty of this work is in obtaining the exact asymptotic estimate of the solutions of certain nonlinear elliptic problems involving p-q Laplace operator which is given in Theorem \ref{thm:sing_esti}. }

For the rest of the paper we denote $d(x):=d(x, \pa \Om),$ the distance function in $\Om$. 
\begin{theorem}
	\label{thm:sing_esti}
	Let $u_\la$ be the unique positive solution of the boundary value problem
	\begin{equation}
	\begin{aligned}
		\label{eqn:sing_esti}
		-\De_p u-\De_q u&=\la \frac{g(x)}{u^\beta} \ \text{in} \ \Om \\
		u&=0\ \text{on}\ \pa \Om
	\end{aligned}
	\end{equation}
	where $1<p<q<\infty, $ $ 0<\beta <1 , \la>0$ and $g$ satisfies the hypothesis $(G_\beta)$ given below:
	\begin{itemize}
    \item[\textbf{($G_\beta$)}]  For a given $\beta\in [0,1), \, \exists$  $\delta\in (0,1-\beta)$ and $C>0,$  such that $ 0< g(x)\leq C d(x)^{-\delta} \mbox{ in } \Omega$ 
\end{itemize}
Then for a given $\la_0>0,$ there exists constants $c_1, c_2 >0$ such that 
	\begin{equation} 
		c_1\la^{\frac{1}{q-1+\beta}}d(x)\leq u_\la(x)\leq c_2 \la^{\frac{1}{q-1+\beta}}d(x)
	\end{equation}
	for all $\la\geq \la_0$ and for every $x \in \Om$.
\end{theorem}
\vspace{0.2 cm}
The existence, uniqueness and H\"{o}lder regularity of (\ref{eqn:sing_esti}) is proved in \cite{giacomoni2021sobolev}. In this work, we are interested in the asymptotic behavior of solution $u_\lambda$ when $\lambda$ tends to infinity. This result is proved in Theorem \ref{thm:sing_esti}, where the main idea of the proof is  inspired by the discussion on scaling property of the solutions of p-q Laplace operator given in \cite{marano2017some}. The proof inevitably depends on the $C^{1,\alpha}$ regularity result proved by Giacomoni et. al in \cite{giacomoni2021sobolev}. As mentioned in the abstract, the estimate of this type is known to be of great importance in establishing several qualitative properties of the solutions of $(P_\la).$ In addition, the ideas presented here can also be extended to certain class of sign changing weight functions.  In an ongoing work, we derive similar estimates to prove the existence of a solution to a nonhomogeneous boundary value problem with an indefinite singular reaction term.




In this article, next we focus our attention to the quasilinear nonlinear elliptic problems of the type 
\begin{equation}
\begin{aligned}
    \label{eqn:uniqueness_pq}
		-\De_p u-\De_q u=\la \frac{f(u)}{u^\beta} \ \text{in} \ \Om \\
		u>0 \, \text{ in } \Omega\\
		u=0\ \text{on}\ \pa \Om
\end{aligned}
\end{equation}
 Recently, existence and multiplicity of solutions for p-q Laplace equation with singular non-linearity has gained considerable attention. Arora(\cite{arora2021multiplicity}) and Acharya et. al.  (\cite{acharya2021existence}) proved the existence of two solutions of \eqref{eqn:uniqueness_pq} using the technique of sub-super solutions. In a series of papers (\cite{papageorgiou2021existence}, \cite{papageorgiou2021positive}, \cite{papageorgiou2021singular}), Papageorgiou and Winkert have explored bifurcation type results describing the changes in the set of positive solutions of the problem as the parameter $\la$ varies. The reaction term considered in their papers has the combined effects of the singular term as well as a super-diffusive growth term. In \cite{papageorgiou2020nonlinear}, Papageorgiou et.al have also provided a similar bifurcation type result for more general non-homogeneous differential operators. These works deal with existence and multiplicity results for (\ref{eqn:uniqueness_pq}), whereas we are interested in the uniqueness aspect of the solution for large $\lambda.$ In this context, the authors in \cite{chu2020uniqueness} and \cite{cong2022uniqueness} have proved the uniqueness of solutions of \eqref{eqn:uniqueness_pq}  when $p=q$ and with certain additional assumptions on $f. $ In section 3 of this paper, we provide a concrete application of Theorem \ref{thm:sing_esti} in proving a uniqueness result for an elliptic boundary value problem (\ref{eqn:uniqueness_pq}), the result is stated below.

\begin{theorem}
	\label{thm: uniqueness_pq}
Let $1<p<q\leq 2$ and consider the elliptic boundary value problem (\ref{eqn:uniqueness_pq}) where
	f satisfies the conditions
	\begin{itemize}
		\item[(A1)] $f:[0,\infty) \rightarrow (0,\infty)$ is of class $C^{1}$ with $\inf_{[0,\infty)}f>0$, and 
		\item[(A2)] There exists a constant $a>0$ such that $\frac{f(z)}{z^{\beta}}$ is decreasing on $[a,\infty)$.
	\end{itemize}
	Also, we assume either $0<\beta < \frac{(q-1)(1+p-q)}{1+q-p}$ or $\sup_{[0,\infty)}f< \infty$ and $0<\beta<1-q+p$. Then there exists $\la_0>0$ such that (\ref{eqn:uniqueness_pq}) has a unique solution for every $\la>\la_0$.
\end{theorem}
\noi The above theorem extends the result proved for p Laplace equation in \cite{cong2022uniqueness}. Our paper is organized as follows. The proof of main results, Theorems 1.1 and 1.2 are detailed in sections 2 and 3 respectively. In the appendix, we prove the existence of a solution for the problem \eqref{eqn:uniqueness_pq} satisfying the conditions $(A1)$ and $(A2)$. We also prove an $L^\infty$ regularity result in the appendix.
\section{Estimates}
 Before proving main results, we state a version of the regularity theorem (Theorem 1.7, \cite{giacomoni2021sobolev}) which is used several times in our paper to obtain uniform $C^{1,\alpha}$ bounds for weak solutions. We say $g$ satisfies the condition $(G_0)$ if $0<g(x)\leq C d(x)^{-\delta}$ for some $C>0$ and $\delta\in (0,1).$
\begin{theorem}
	\label{thm:giacomoni_pq_regularity}
	Let $u \in W_{0}^{1,q}(\Om)$ be the weak solution of the BVP
	\begin{equation}
	\begin{aligned}
		-\mu \De_p u-\De_q u =g(x)  \ \text{in} \ \Om \\
		u=0\ \text{on}\ \pa \Om
	\end{aligned}
	\end{equation}
	where $1<p<q<\infty $, $0\leq\mu\leq \mu_0$ for some $\mu_0>0$ and $g$ satisfies the condition $(G_0).$  Suppose $0\leq u\leq M$ and $0\leq u \leq Kd(x)$ in $\Om$ for some constants $M$ and $K$. Then there exists a constant $\alpha \in (0,1)$ depending only on $N,q,\de$ such that $u \in C^{1,\alpha}(\Bar{\Om})$ and
	\begin{equation}
		\|u\|_{C^{1,\alpha}(\Bar{\Om})}\leq C(N,q,\de,M,\Om).
	\end{equation}
\end{theorem}
Now, we state and prove a proposition which provides the case $\beta=0$ of Theorem \ref{thm:sing_esti}.
\begin{proposition}
	\label{thm: constant_esti}
	Let $v_\la$ be the unique solution of the following quasilinear BVP:
	\begin{equation}
	\begin{aligned}
	\label{vlambda}
		-\De_p v-\De_q v=\la g(x) \ \text{in} \ \Om \\
		v=0\ \text{on}\ \pa \Om
	\end{aligned}
	\end{equation}
	where $1<p<q<\infty$, $\la>0, $  and $g$ satisfies the hypothesis $(G_0).$  Then for a given $\la_0>0,$ there exists positive constants $ c_1,c_2$ such that
	\begin{equation*}
		c_1\la^{\frac{1}{q-1}}d(x)\leq v_\la(x)\leq c_2 \la^{\frac{1}{q-1}}d(x)
	\end{equation*}
	for all $\la\geq \la_0$ and for every $x \in \Om$.
\end{proposition}

\noi \textbf{Proof: } The existence and uniqueness of weak solution $v_\la\in W^{1,q}_0(\Omega)$ of (\ref{vlambda}) can be obtained by standard minimization technique. Next, we can show that $\la\rightarrow v_\la$ is monotone, namely: 
	$$  v_{\la_1}(x)\leq v_{\la_2}(x) \; \mbox{ if } \; \la_1\leq \la_2$$
	by taking $(v_{\la_1}-v_{\la_2})^+$ as the test function in the weak formulation. Motivated by the discussions in the work of  Marano and Mosconi \cite{marano2017some},  we define :
	\begin{equation}
	    \tilde{v}_{\la}:=\la^{\frac{-1}{q-1}}v_\la 
	\end{equation} 
	and prove the asymptotic behavior of the solution $\tilde{v}_\lambda$ for large $\lambda.$ Using this substitution, the boundary value problem (\ref{vlambda}) reduces to 
	\begin{equation}
	\begin{aligned}
		\label{eqn:theorem 1}
		-\mu \De_p \tilde{v}_{\la}-\De_q \tilde{v}_{\la}=g(x)\ \text{in}\ \Om \\
		\tilde{v}_{\la}=0\ \text{on}\ \pa \Om
	\end{aligned}
	\end{equation}	
	where $\mu=\la^{\frac{p-q}{q-1}}$. It is important to note that $\mu\rightarrow 0$  as $\la$ tends to infinity. 
	We note that by Remark \ref{lastrem} of Appendix, $\|\tilde{v}_\lambda\|_{\infty}$ is uniformly bounded as $\lambda\rightarrow \infty.$ Thanks to Proposition 2.7 of \cite{giacomoni2021sobolev}, now we can apply Theorem \ref{thm:giacomoni_pq_regularity} and obtain that $\{\tilde{v}_{\la}\}$ is uniformly bounded in $C^{1,\alpha}(\Bar{\Om})$ for some $\alpha \in (0,1)$. Thus by Ascoli-Arzela theorem, upto a subsequence  \begin{equation} 
	\tilde{v}_\lambda \rightarrow v_0 \mbox{ in } C^1_0(\overline{\Omega}) \mbox{ as } \lambda\rightarrow \infty.
	\end{equation}
 Now, for any $\phi \in C_{c}^{\infty}({\Om})$,
	\begin{equation*}
		\mu \int_{\Om} |\nabla \tilde{v}_{\la}|^{p-2}\nabla \tilde{v}_{\la}\boldsymbol{\cdot}\nabla \phi  dx+ \int_{\Om} |\nabla \tilde{v}_{\la}|^{q-2}\nabla \tilde{v}_{\la}\boldsymbol{\cdot}\nabla \phi  dx = \int_{\Om}\phi g dx
	\end{equation*}
	Since $\nabla \tilde{v}_{\la} \rightarrow \nabla v_0$ uniformly in $\Omega,$ applying the limit as $\la \rightarrow \infty$ on either sides we get 
	\begin{equation*}
		\int_\Om |\nabla v_0|^{q-2}\nabla v_0 \boldsymbol{\cdot} \nabla \phi dx = \int_{\Om}\phi g dx
	\end{equation*}
	This implies that $v_0$ is the weak solution of 
	\begin{equation}
	\begin{aligned}
		-\De_q v_0=g(x) \ \text{in}\ \Om  \\
		v_0 =0 \ \text{on}\ \pa \Om
	\end{aligned}
	\end{equation}
	By the uniqueness of weak solution $v_0,$ the original sequence $\tilde{v}_\la$ itself converges to $v_0$ in $C^1_0(\overline{\Omega}).$ Let $\nu$ represents the unit outward normal on $\pa \Om$ and $m=\max_{x\in \pa \Om} \frac{\pa v_0}{\pa \nu}(x).$ We know that $m<0,$ thanks to Theorem 5 of Vasquez  \cite{vazquez1984strong}. By uniform convergence of $\frac{\pa \tilde{v}_\la }{\pa \nu} \rightarrow \frac{\pa v_0}{\pa \nu}$ we can find a $\la'>0$ for which  
	 $ \displaystyle \frac{\pa \tilde{v}_\la }{\pa \nu} (x) <\frac{m}{2} \;\; \forall \; x \in \pa\Omega \mbox{ and for all } \la\geq \la'.$
 Since $\Om$ is  assumed to be a smooth bounded domain in $\mathbb{R}^N$, there exists a $c>0$ independent of $\la \geq \la'$ such that 
	\begin{equation}\label{lower}
		\tilde{v}_\la (x)\geq c d(x) \mbox{ for } \la \geq \la' \mbox{ and } \forall x \in \Omega.
	\end{equation}
Once again using Theorem \ref{thm:giacomoni_pq_regularity} for some constant $C>0$ 
	\begin{equation}\label{upper}
		\tilde{v}_{\la}(x)\leq C \; d(x) \mbox{ for } \la \geq \la' \mbox{ and } \forall x \in \Omega.
		\end{equation}
Finally, using the monotonicity of $v_\la$ in any compact sub interval $[a,b]$ of $(0,\infty)$ there exists positive constants $m_1,m_2$ such that $m_1 d(x) \leq v_\la(x) \leq m_2 d(x)$ for all $\la\in [a,b].$ Combining this property of $v_\lambda$ along with the lower and upper estimates \eqref{lower}, \eqref{upper} for $\tilde{v}_\la$ we infer that for any given $\la_0>0$ there exists constants $c_1, c_2$ (which depends only on $\la_0$) such that 
$$ 	c_1\la^{\frac{1}{q-1}}d(x)\leq v_\la(x)\leq c_2 \la^{\frac{1}{q-1}}d(x).$$
Hence the proof. 
\hfill\qed\\

We use the above proposition to prove our main result of this paper. 
\begin{proof}[\textbf{Proof of Theorem \ref{thm:sing_esti}}]
	 We define $\tilde{u}_{\la}:=\la^{\frac{-1}{q-1+\beta}}u_\la$. Then, clearly $\tilde{u}_{\la}\in W_0^{1,q}(\Om)$ is the unique weak solution of the boundary value problem 
	\begin{equation}
	\begin{aligned}
		\label{eqn:sing_l infinity}
		-\gamma \De_p \tilde{u}_{\la}-\De_q \tilde{u}_{\la}=\frac{g(x)}{(\tilde{u}_{\la})^\beta} \ \text{in} \ \Om  \\
		\tilde{u}_{\la}=0\ \text{on} \ \pa \Om  
	\end{aligned}
	\end{equation}
	where $\gamma = \la^{\frac{p-q}{q-1+\beta}}$. As before, we note that $\gamma\rightarrow 0$ as $\lambda\rightarrow \infty$ and vice versa. Now, using the Moser iteration technique we can prove that $\|\tilde{u}_{\la}\|_{L^\infty}\leq M_0$ for some $M_0>0$ independent of $\gamma.$ The details of the uniform $L^\infty$ estimate is proved in the appendix (Theorem \ref{linf}). Therefore,
	we have $\frac{g(x)}{(\tilde{u}_{\la})^\beta}\geq \frac{g(x)}{M_0^\beta}$. 
	By the weak comparison principle, $\tilde{u}_{\la}\geq w_\gamma$ where $w_\gamma$ is the unique weak solution of 
	\begin{eqnarray}
		-\gamma \De_p w_\gamma -\De_q w_\gamma= \frac{g(x)}{M_0^\beta} \ \text{ in }  \Om \nonumber \\
		w_\gamma=0\ \text{on} \ \pa \Om \nonumber
	\end{eqnarray}
	Following the ideas in the proof of Proposition \ref{thm: constant_esti}, we find a $c_1>0$ such that $w_\gamma \geq c_1 d(x)$ for all $\gamma\in (0,\gamma_0).$  This implies, for large $\lambda$  \begin{equation}\label{lbb} 
	u_\la(x) \geq c_1 \la^{\frac{1}{q-1+\beta}}d(x)\mbox{ for all } \lambda \mbox{ large and for every }x \in \Om. \end{equation} Next, to obtain an upper bound for $u_\lambda$ we note that 
	\begin{equation*}
		-\De_p u_\la -\De_q u_\la =\la\frac{g(x)}{u_\la^\beta}\leq \la^{\frac{q-1}{q-1+\beta}}\frac{g(x)}{{c_1^\beta d(x)^\beta}}
	\end{equation*}
	From the assumption $(G_\beta)$, we have  ${g(x)}{d(x)^{-\beta}}\leq C d(x)^{-(\beta+\delta)}$ and $\beta+\delta <1.$ Now, let $z_\lambda$ denote the unique solution of
	\begin{eqnarray}
	 -\Delta_p z_\la-\De_q z_\la= \la^{\frac{q-1}{q-1+\beta}}{c_1^{-\beta} g(x) }{{ d(x)^{-\beta}}} \text{ in }  \Om \nonumber \\
		z_\lambda=0\ \text{ on }  \pa \Om \nonumber.
	\end{eqnarray}
	Again, from Proposition \ref{thm: constant_esti} there exists $c_2>0$ such that 
	$z_\la(x)\leq c_2 \la^{\frac{1}{q-1+\beta}} \,d(x)$ for large $\lambda.$ 
	Clearly, by weak comparison principle,  
	\begin{equation}\label{ubb}
	u_\lambda(x)\leq z_\la(x) \leq c_2 \la^{\frac{1}{q-1+\beta}} d(x) \mbox{ for large } \lambda.
	\end{equation}
	If $\la_1\leq \la_2,$ then taking $(u_{\la_1}-u_{\la_2})^+$ as the test function in the weak formulation of the \eqref{eqn:sing_esti} we can prove that $u_{\la_1}\leq u_{\la_2}$ a.e in $\Omega.$ Now using monotonicity of $u_\la$ along with the estimates (\ref{lbb}) and (\ref{ubb}) we have the required result. 
\end{proof}

\section{Uniqueness result}
In this section we consider the singular elliptic boundary value problem 
	\begin{equation}
	\begin{aligned}
		\label{eqn:uniqueness_pq1}
		-\De_p u-\De_q u=\la \frac{f(u)}{u^\beta} \ \text{in} \ \Om \\
		 u > 0 \text{ in } \Omega\\
		u=0\ \text{on}\ \pa \Om
	\end{aligned}
	\end{equation}
and	f satisfies the conditions (A1) and (A2). The existence of at least one positive solution of (\ref{eqn:uniqueness_pq1}) can be proved using the technique of monotone iterations (see Appendix for a proof). In this section we only prove the uniqueness of its weak solution when $\lambda$ is large,  i.e. Theorem \ref{thm: uniqueness_pq}.  The proof crucially depends on the estimates for (\ref{eqn:sing_esti}) when $\la >>1$. \\[3mm]
\textbf{Proof of Theorem 1.2 : }
	We prove the theorem in three steps. It may be noted that the steps 1 and 2 hold true for any $1< p< q<\infty .$ Restriction on $1<p<q<2$ is essentially due to the step 3.\\
	\textbf{Step 1}: \textit{There exists $k_1>0$ such that 
		\begin{equation}
			u(x) \geq k_1\la^{\frac{1}{q-1+\beta}}d(x) \mbox{ for large } \lambda.
		\end{equation}}
	Let $c= \inf_{[0,\infty)}f,$ then $c>0$ by $(A1)$. This implies, by weak comparison principle $u\geq v$ where $v \in W_0^{1,q}(\Om)$ is the unique weak solution of
	\begin{eqnarray}
		-\De_p v -\De_q v = \frac{\la c}{v^\beta}\ \text{in} \ \Om \nonumber \\
		v=0 \ \text{on} \ \pa \Om.\nonumber
	\end{eqnarray}
	By Theorem \ref{thm:sing_esti} above, we know that there exists $k_1$ such that
	\begin{equation*}
		\label{eqn:lowerbound_dist}
		v(x)\geq k_1\la^{\frac{1}{q-1+\beta}}d(x)	
	\end{equation*} 
	 for all $\la\geq\la_0$. Hence, $u \geq k_1\la^{\frac{1}{q-1+\beta}}d(x)$ for large $\la$.\\
	\vspace{0.2 cm}\\[3mm]
	\textbf{Step 2}: \textit{There exists $k_2>0$ such that $\|u\|_{C^1(\Bar{\Om})}\leq k_2 \la^{\frac{1}{q-1}}.$ In addition if we assume that $\sup_{[0,\infty)}f< \infty,$ then $\|u\|_{C^1(\Bar{\Om})}\leq k_2 \la^{\frac{1}{q-1+\beta }}.$}\\
	Let us define $w=\la^{-\frac{1}{(q-1)}}u$, then for $\delta=\la^{\frac{p-q}{q-1}}$
	\begin{equation}\label{eqnw}
		-\delta \Delta_{p}w - \Delta_{q} w = \frac{f(\la^{\frac{1}{q-1}}w)}{(\la^{\frac{1}{q-1}}w)^\beta}
	\end{equation}
	By (A1)-(A2) and Step 1 we can estimate the RHS of above equation as shown below. 
	\begin{equation*}
	\frac{f(\la^{\frac{1}{q-1}}w)}{(\la^{\frac{1}{q-1}}w)^\beta} \leq \; K(1+\frac{1}{(\la^{\frac{1}{q-1}}w)^\beta}) \leq \; \frac{K_1}{d^\beta}.
	\end{equation*}
	Now we can apply Theorem 2.1 to (\ref{eqnw}) and obtain a constant $k_2$ independent of $\delta$ such that  $\|w\|_{C^1(\Bar{\Om})}\leq k_2.$ That is, 
	$\|u\|_{C^1(\Bar{\Om})}\leq k_2 \la^{\frac{1}{q-1}}$.
	
	In addition if we suppose that $\sup_{[0,\infty)}f<\infty$, then we obtain a sharper estimate for the $C^1$ norm of $u.$ Define $\tilde{w}:=\la^{\frac{-1}{q-1+\beta}}u$. Then for $\gamma=\la^{\frac{p-q}{q-1+\beta}},$ 
	\begin{equation*}
		-\gamma \Delta_{p}\tilde{w} - \Delta_{q} \tilde{w} = \frac{f(\la^{\frac{1}{q-1+\beta}}\tilde{w})}{\tilde{w}^\beta}\leq \frac{\sup_{[0,\infty)} f}{\tilde{w}^\beta}
	\end{equation*}
	Using the estimate proven in Step 1 above, one can verify that $\tilde{w}$ satisfies the hypothesis of Theorem \ref{thm:giacomoni_pq_regularity} and hence $\tilde{w}$ is uniformly bounded in $C^{1,\alpha}(\Bar{\Om})$ independent of $\gamma$. Hence, there exists $k_2>0$ such that $\|\tilde{w}\|_{C^1(\Bar{\Om})}\leq k_2$. In other words, $\|u\|_{C^1(\Bar{\Om})}\leq k_2\la^{\frac{1}{q-1+\beta}}$.\\[3mm]
	\textbf{Step 3}: \textit{Uniqueness of solution for \eqref{eqn:uniqueness_pq1} when $\lambda$ is large is proved using Step 1 and Step 2}. \\
	On the contrary, suppose that the solution of (\ref{eqn:uniqueness_pq1}) is not unique. Let $u, v \in C^{1, \alpha}(\Bar{\Om})$ be weak solutions of $(\ref{eqn:uniqueness_pq1}).$ Now,
	\begin{equation*}
		-\Delta_p u-\Delta_q u-(-\Delta_p v-\Delta_q v)=\la(h(u)-h(v)).
	\end{equation*}
	where $h(z)=\frac{f(z)}{z^\beta}$. Multiplying both sides by $u-v$ and integrating, 
	\begin{eqnarray}
	\int_{\Omega}((|\nabla u|^{p-2}\nabla u-|\nabla v|^{p-2}\nabla v) + (|\nabla u|^{q-2}\nabla u-|\nabla v|^{q-2}\nabla v))\boldsymbol{\cdot}(\nabla u-\nabla v)dx \nonumber \\
		\;\;\;\;\;\;= \;\;\;\;\;\; \la \int_{\Omega}(h(u)-h(v))(u-v)dx \nonumber 
	\end{eqnarray}
	By the inequality
	\begin{equation*}
		(|a|^{r-2}a -|b|^{r-2}b)\boldsymbol{\cdot}(a-b)\geq (r-1)\frac{|a-b|^2}{(|a|+|b|)^{2-r}} \mbox{ for } 1<r\leq 2
	\end{equation*}
where $a,b\in \mathbb{R}^n$ and using the fact that $|\nabla u(x)|+|\nabla v(x)| \leq \|u\|_{C^1}+ \|v\|_{C^1}$ we obtain, 
	\begin{equation*}
		(p-1)\int_{\Omega}\frac{|\nabla u-\nabla v|^2}{(\|u\|_{C^1}+\|v\|_{C^1})^{2-p}} dx +(q-1)\int_{\Omega}\frac{|\nabla u-\nabla v|^2}{(\|u\|_{C^1}+\|v\|_{C^1})^{2-q}} dx \leq \la \int_{\Omega}h'(\xi)(u-v)^2dx  .\\[3mm]
	\end{equation*}
	We have evaluated the RHS using mean value theorem on $h$,  for some function $\xi$ lying in between $u$ and $v$.
	By the estimate in Step 2, we have $\|u\|_{C^1}+\|v\|_{C^1}\leq 2k_2\la^{\frac{1}{q-1}}$. Also, as $p<q$ and $\la$ is large,
	\begin{equation*}
		(2k_2\la^{\frac{1}{q-1}})^{2-p}\geq (2k_2\la^{\frac{1}{q-1}})^{2-q}.  
	\end{equation*} 
	Hence, 
	\begin{equation}\nonumber 
		(p-1)\int_{\Omega}|\nabla (u-v)|^2 dx + (q-1)\int_{\Omega}|\nabla (u-v)|^2 dx \leq k\la^{1+\frac{2-p}{q-1}}\int_{\Omega}h'(\xi)(u-v)^2dx 
	\end{equation} 
	for $k>0$. That is, 
	\begin{equation}
		(p+q-2)\int_{\Omega}|\nabla (u-v)|^2 dx + \leq k\la^{1+\frac{2-p}{q-1}}\int_{\Omega}h'(\xi)(u-v)^2dx 
	\end{equation} 
	
	Now, we closely follow the proof of Theorem 1.1 of \cite{cong2022uniqueness} and using the estimates in Step 1 and Step 2, for any $f$ satisfying $(A1)$ and $(A2)$ we obtain that 
	\begin{itemize}
		\item [(i)] \begin{equation}\nonumber
			(p+q-2)\int_{\Omega}|\nabla (u-v)|^2 dx  \leq m \la^{1+\frac{2-p}{q-1}-\frac{2}{q-1+\beta}} \int_\Omega |\nabla(u-v)|^2 dx.
		\end{equation}
		In addition, if we assume $\sup_{[0,\infty)}f<\infty,$
		\item[(ii)]\begin{equation}\nonumber
			(p+q-2)\int_{\Omega}|\nabla (u-v)|^2 dx \leq m \la^{1-\frac{p}{q-1+\beta}}\int_\Omega |\nabla(u-v)|^2 dx.
		\end{equation}

	\end{itemize}
	It is clear that $\la$ in the RHS has a negative exponent when $\beta<\beta'=\frac{(q-1)(1+p-q)}{1+q-p}$ in (i) and when $\beta<\beta''=1-q+p$ in (ii). Clearly, $\beta', \beta'' \in (0,1)$.
\par Thus for $\lambda$ large, the inequalities $(i)$ and $(ii)$ holds only when the integrand $\nabla(u-v)$ is identically zero. Since $u=v=0$ on $\pa\Omega,$ this implies that $u\equiv v$ in $\Om$ for large $\lambda.$ Hence uniqueness results holds for $\la$ large. \hfill\qed.

\section{Appendix}
In this section we shall prove the existence result for (\ref{eqn:uniqueness_pq}) and an $L^\infty$ regularity result for (\ref{eqn:sing_l infinity}). 
\subsection{Existence of weak solution for (\ref{eqn:uniqueness_pq})}
\begin{proposition}
	Suppose all the conditions of Theorem \ref{thm: uniqueness_pq} are satisfied. Then, there exists $u_\la \in W_0^{1,p}(\Om)$ such that $u_\la$ is a weak solution of (\ref{eqn:uniqueness_pq}).
\end{proposition}
\textbf{Proof: } 
	The differential equation
	\begin{equation*}
		-\De_p u -\De_q u = \la \frac{f(u)}{u^\beta}
	\end{equation*}
	can be written as 
	\begin{equation}
		\label{eqn:transform_existence}
		-\De_p u -\De_q u-\la \frac{f(0)}{u^\beta}=\la h(u)
	\end{equation}
	where $h(u)=\frac{f(u)-f(0)}{u^\beta}$. We assume that the function $h(u)$ is monotonically increasing. If not, we can choose an appropriate $K>0$ such that  $h(u)+Ku$ is monotone increasing in $[0,\infty).$ And we will consider the PDE  $	-\De_p u -\De_q u-\la \frac{f(0)}{u^\beta}+K u=\la (h(u)+K u)$ instead of (\ref{eqn:transform_existence}). The proof does not vary much among both cases. So, without loss of generality, the monotonicity of $h$ can be assumed.
\par Let  $c=\inf_{[0,\infty)} f(t)$, then $c>0$ by the condition $(A1)$ and  	\begin{equation*}
		h(t)\geq \; (\frac{c-f(0)}{t^\beta})
	\end{equation*}
	for all $t>0$. Assume that $\underline{u} \in W_0^{1,q}(\Om)$ is the unique weak solution of 
	\begin{equation}
	\begin{aligned}
		\label{eqn:subsolution}
		-\De_p u -\De_q u=\la \frac{c}{u^\beta} \ \text{in}\ \Om  \\
		u=0\ \text{on}\ \pa \Om.
	\end{aligned}
	\end{equation}
	The existence and uniqueness of $\underline{u}$ is known from \cite{giacomoni2021sobolev}. Clearly, $\underline{u}$ is a subsolution of (\ref{eqn:transform_existence}). By (A2), there exists $C>0$ such that
\begin{equation*}
	\frac{f(t)}{t^\beta} \leq C(1+\frac{1}{t^\beta})
\end{equation*}
for all $t>0$. Let $\Bar{u} \in  W_0^{1,q}(\Om)$ be the unique weak solution of 
\begin{equation}
\begin{aligned}
	-\De_p u -\De_q u=\la C(1+\frac{1}{u^\beta}) \ \text{in}\ \Om  \\
	u=0\ \text{on}\ \pa \Om.
\end{aligned}
\end{equation}
$\Bar{u}$ is a super-solution of (\ref{eqn:transform_existence}). The existence, uniqueness and $L^\infty$ regularity of $\Bar{u}$ can be derived from \cite{giacomoni2021sobolev}. Choosing $C$ large enough if required we can show that $\underline{u} \leq \overline{u}.$ Let $v_0=\underline{u}$. Define the sequence $\{v_n\}_{n \in \mathbb{N}}\subset  W_0^{1,q}(\Om)$ iteratively as follows:
Let $v_{n+1}$ be the unique weak solution of the BVP
\begin{equation}
 \begin{aligned}   
\label{eqn:iteration_existence_pq}
	-\De_p v_{n+1} -\De_q v_{n+1}-\la \frac{f(0)}{v_{n+1}^\beta}&=&\la h(v_n)  \ \text{in}\ \Om  \\
	v_{n+1}&=&0\ \text{on}\ \pa \Om.
\end{aligned}
\end{equation}
for every $n \in \mathbb{N}.$ 
Using the monotonicity of $h(t)$ we can show that 
$$c d(x) \leq \underline{u}\leq \cdots v_n \leq v_{n+1}\leq \cdots \Bar{u}\leq M .$$
Now using the standard ideas we can pass through the limit in \eqref{eqn:iteration_existence_pq} and prove the existence of a weak solution of \eqref{eqn:uniqueness_pq}.
\hfill\qed.

\subsection{Uniform $L^\infty$ regularity} 
 Now, we prove that $\{\tilde{u}_\la\}$ is uniformly bounded in $L^\infty(\Om)$, where $\tilde{u}_\la$ is the unique weak solution of 
\begin{equation}
\begin{aligned}
	\label{eqn:uniform_l_infinity}
	-\gamma \De_p \tilde{u}_{\la}-\De_q \tilde{u}_{\la}=\frac{g(x)}{(\tilde{u}_{\la})^\beta} \ \text{in} \ \Om  \\
	\tilde{u}_{\la}=0\ \text{on} \ \pa \Om  
\end{aligned}
\end{equation}
where $1<p<q<\infty$, $0<\beta<1,$ and $0<g(x)\leq \frac{C}{d(x)^\de}$ for some $C>0$, $0<\beta+\de<1$ and $\gamma = \la^{\frac{p-q}{q-1+\beta}}$. 
We now state the uniform $L^\infty$ regularity theorem:
\begin{theorem}\label{linf}
	Let $\tilde{u}_{\la}$ be the unique weak solution of (\ref{eqn:uniform_l_infinity}). Then there exists $M_0>0$ independent of $\lambda$ such that $\|\tilde{u}_{\la}\|_{L^\infty(\Om)}\leq M_0$ for every $\la>0$. 
\end{theorem}
\begin{proof}
	The proof follows the ideas of Lemma 3.2 from \cite{kumar2020singular} and Theorem E.0.19 of \cite{peral1997multiplicity}. Let $u \in W_0^{1,q}(\Om)$ be the weak solution of (\ref{eqn:uniform_l_infinity}) and $\phi$ be a $C^1$ cut-off function such that $\phi(t)=0$ for $t\leq 0$, $\phi'(t)\geq 0$ for $0\leq t\leq 1$ and $\phi(t)=1$ for $t\geq 1$. By the weak comparison principle we know that $u\geq 0$ in $\Om.$ Let us  define $\phi_\epsilon(t):=\phi(\frac{t-1}{\epsilon})$ for $t\in \mathbb{R}$. Then, $\phi_\epsilon(u)\in W^{1,q}_0(\Omega)$ and  $\nabla \phi_\epsilon(u)=\phi_\epsilon'(u)\nabla u.$ 
	For a non-negative function $w \in C_c^{\infty}(\Om),$ using $\phi_\epsilon(u)w$ as a test function in the weak formulation of (\ref{eqn:uniform_l_infinity}) we obtain
	\begin{eqnarray}
		\int_\Om (\gamma |\nabla u|^{p} + |\nabla u|^{q}) \phi_\epsilon'(u)w dx + \int_\Om (\gamma |\nabla u|^{p-2}+ |\nabla u|^{q-2} ) \phi_\epsilon(u)\nabla u \boldsymbol{\cdot} \nabla w dx \nonumber \\
		\ \ \ = \int_\Om \frac{g(x)}{u^\beta} \phi_\epsilon(u)w dx \nonumber 
	\end{eqnarray}	
	Taking $\epsilon \rightarrow 0$, we have 
	\begin{equation}\nonumber
		\int_{\Om \cap \{u\geq 1\}} (\gamma |\nabla u|^{p-2}+ |\nabla u|^{q-2} ) \nabla u \boldsymbol{\cdot} \nabla w dx \leq \int_{\Om\cap \{u\geq 1\}} \frac{g(x)}{u^\beta} wdx \leq \int_\Om g(x) wdx 
	\end{equation}	
	as $\phi'(t)\geq 0$ for all $t$ and $\frac{1}{u^\beta}\leq 1$ when $u \geq 1$. That is,
	\begin{equation}\nonumber
	\gamma \int_\Om |\nabla \Bar{u}|^{p-2} \nabla \Bar{u}\boldsymbol{\cdot} \nabla w dx	+\int_\Om |\nabla \Bar{u}|^{q-2} \nabla \Bar{u}\boldsymbol{\cdot} \nabla w dx \leq C\int_\Om \frac{1}{d(x)^\delta} w dx
	\end{equation}
	as $\gamma >0$, where $\Bar{u}:=(u-1)_+$ is the positive part of the function $(u-1).$
 It is clear that $\frac{C}{d(x)^\delta} \in W^{-1,r}(\Om)$ for $1<r<\infty $. So, there exists a function $\textbf{F}=(F_1,F_2,..,F_N)$ such that $\frac{C}{d(x)^\delta}=\text{div}(F)$. Here, $F_i \in L^r(\Om)$ for $1\leq i\leq N$.
	\begin{equation}
		\label{eqn:l infinity_pq_for test function}
		\gamma \int_\Om |\nabla \Bar{u}|^{p-2} \nabla \Bar{u}\boldsymbol{\cdot} \nabla w dx	+ \int_\Om |\nabla \Bar{u}|^{q-2} \nabla \Bar{u}\boldsymbol{\cdot} \nabla w dx \leq \int_\Om \textbf{F}\boldsymbol{\cdot} \nabla w dx 
	\end{equation} 
\par Define the truncation function $T_k(s):=(s-k)\chi_{[k,\infty)}$ which was introduced in \cite{stampacchia1966equations}, for $k>0$. Let $U_k:=\{x \in \Om : \Bar{u}(x)\geq k\}$ for $k>0$.  Choosing $T_k(\Bar{u})$ for a $k>0$ as the test function in (\ref{eqn:l infinity_pq_for test function}) and using the fact that $\gamma>0,$ we have 
	\begin{equation*}
		\int_{U_k} |\nabla \Bar{u}|^{q} dx \leq \int_{U_k}\textbf{F}\boldsymbol{\cdot} \nabla \Bar{u} dx. 
	\end{equation*} 
	So,
	\begin{equation}
		\Big(\int_{U_k}|\nabla  \Bar{u}|^q dx\Big)^{1-\frac{1}{q}}\leq \Big(\int_{U_k}|\textbf{F}|^r dx\Big)^\frac{1}{r}|U_k|^{1-(\frac{1}{r}+\frac{1}{q})}
	\end{equation}
	by H{\"o}lder's inequality, where $|U_k|$ is the Lebesgue measure of the set $U_k$. We can now proceed exactly as in \cite{peral1997multiplicity} and obtain the required result.

\end{proof}	
\begin{remark}\label{lastrem}
The above theorem is true even when $\beta=0.$
\end{remark}	
\bibliographystyle{plain}
\bibliography{References}	

\begin{thebibliography}{10}

\bibitem{acharya2021existence}
Ananta Acharya, Ujjal Das, and Ratnasingham Shivaji.
\newblock Existence and multiplicity results for pq-laplacian boundary value
  problems.
\newblock {\em Electronic Journal of Differential Equations, Special Isuue
  01(2021), pp. 293-300}, 2021.

\bibitem{arora2021multiplicity}
Rakesh Arora.
\newblock Multiplicity results for nonhomogeneous elliptic equations with
  singular nonlinearities.
\newblock {\em Communications on Pure and Applied Analysis}, 21(6):2253, 2022.

\bibitem{cherfils2005stationary}
Laurence Cherfils and Yavdat Il'Yasov.
\newblock On the stationary solutions of generalized reaction diffusion
  equations with p,q-laplacian.
\newblock {\em Communications on Pure \& Applied Analysis}, 4(1):9, 2005.

\bibitem{chu2020uniqueness}
KD~Chu, DD~Hai, and R~Shivaji.
\newblock Uniqueness for a class of singular quasilinear dirichlet problem.
\newblock {\em Applied Mathematics Letters}, 106:106306, 2020.

\bibitem{cong2022uniqueness}
P~Cong, D~Hai, and R~Shivaji.
\newblock A uniqueness result for a class of singular p-laplacian dirichlet
  problem with non-monotone forcing term.
\newblock {\em Proceedings of the American Mathematical Society},
  150(02):633--637, 2022.

\bibitem{fife2013mathematical}
Paul~C Fife.
\newblock {\em Mathematical aspects of reacting and diffusing systems},
  volume~28.
\newblock Springer Science \& Business Media, 2013.

\bibitem{giacomoni2021sobolev}
Jacques Giacomoni, Deepak Kumar, and K~Sreenadh.
\newblock Sobolev and h{\"o}lder regularity results for some singular
  nonhomogeneous quasilinear problems.
\newblock {\em Calculus of Variations and Partial Differential Equations},
  60(3):1--33, 2021.

\bibitem{kumar2020singular}
Deepak Kumar, Vicen{\c{t}}iu~D R{\u{a}}dulescu, and K~Sreenadh.
\newblock Singular elliptic problems with unbalanced growth and critical
  exponent.
\newblock {\em Nonlinearity}, 33(7):3336, 2020.

\bibitem{marano2017some}
Salvatore Marano and Sunra Mosconi.
\newblock Some recent results on the dirichlet problem for (p, q)-laplace
  equations.
\newblock {\em Discrete Contin. Dyn. Syst. Ser. S}, (2):279--291, 2018.

\bibitem{papageorgiou2020nonlinear}
Nikolaos~S Papageorgiou, Vicen{\c{t}}iu~D R{\u{a}}dulescu, and Du{\v{s}}an~D
  Repov{\v{s}}.
\newblock Nonlinear nonhomogeneous singular problems.
\newblock {\em Calculus of Variations and Partial Differential Equations},
  59(1):1--31, 2020.

\bibitem{papageorgiou2021existence}
Nikolaos~S Papageorgiou and Patrick Winkert.
\newblock Existence and nonexistence of positive solutions for singular (p,
  q)-equations with superdiffusive perturbation.
\newblock {\em Results in Mathematics}, 76(4):1--20, 2021.

\bibitem{papageorgiou2021positive}
Nikolaos~S Papageorgiou and Patrick Winkert.
\newblock Positive solutions for singular anisotropic (p, q)-equations.
\newblock {\em The Journal of Geometric Analysis}, 31(12):11849--11877, 2021.

\bibitem{papageorgiou2021singular}
Nikolaos~S Papageorgiou and Patrick Winkert.
\newblock Singular dirichlet (p, q)-equations.
\newblock {\em Mediterranean Journal of Mathematics}, 18(4):1--20, 2021.

\bibitem{peral1997multiplicity}
Ireneo Peral.
\newblock Multiplicity of solutions for the p-laplacian.
\newblock {\em International Center for Theoretical Physics Lecture Notes,
  Trieste}, 1997.

\bibitem{stampacchia1966equations}
Guido Stampacchia.
\newblock Equations elliptiques du second ordre {\`a} coefficients discontinus.
\newblock {\em S{\'e}minaire Jean Leray}, (3):1--77, 1966.

\bibitem{vazquez1984strong}
Juan~Luis V{\'a}zquez.
\newblock A strong maximum principle for some quasilinear elliptic equations.
\newblock {\em Applied Mathematics and Optimization}, 12(1):191--202, 1984.

\bibitem{wilhelmsson1987explosive}
Hans Wilhelmsson.
\newblock Explosive instabilities of reaction-diffusion equations.
\newblock {\em Physical review A}, 36(2):965, 1987.

\bibitem{zhykov1986averaging}
V.V Zhykov.
\newblock Averaging of functional of the calculus of variations and elasticity
  theory.
\newblock {\em Izv.Akad.Nauk SSSR, Ser.}, pages 675--710, 1986.

\bibitem{jikov2012homogenization}
V.V Zhykov, Sergei~M Kozlov, and Olga~Arsenievna Oleinik.
\newblock {\em Homogenization of differential operators and integral
  functionals}.
\newblock Springer Science \& Business Media, 2012.

\end{thebibliography}

\end{document}